\documentclass[12pt]{article}
\usepackage{amssymb,amsmath,graphicx,amsthm,mathrsfs}
\usepackage[usenames]{color}

\newtheorem{thm}{Theorem}[section]

\newtheorem{lem}[thm]{Lemma}

\theoremstyle{definition}

\theoremstyle{remark}

\numberwithin{equation}{section}

\title{Singular perturbation of reduced wave equation and scattering from an embedded obstacle}
\author{Hongyu Liu\thanks{Department of
Mathematics and Statistics, University of North Carolina, Charlotte,
NC 28263, USA. ({\tt hliu28@uncc.edu})}\and Zaijiu Shang\thanks{Institute of Mathematics, Academy of Mathematics and
Systems Science, Chinese Academy of Sciences, Beijing 100190, P. R. The work of this author was partially supported by grant under NSF No.10990012. ({\tt zaijiu@amss.ac.cn})} \and  Hongpeng
Sun\thanks{Institute of Mathematics, Academy of Mathematics and
Systems Science, Chinese Academy of Sciences, Beijing 100190, P. R.
China. The work of this author was partially supported by grant under NSF No.10990012. ({\tt hpsun@amss.ac.cn})}\and Jun Zou\thanks{Department of Mathematics, Chinese University of Hong Kong, Shatin, N. T., Hong Kong.
The work of this author was substantially
supported by Hong Kong RGC grants (Projects 405110 and
404611).
({\tt zou@math.cuhk.edu.hk})} }

\begin{document}

\maketitle

\begin{abstract}
We consider time-harmonic wave scattering from an inhomogeneous isotropic medium supported in a bounded domain $\Omega\subset\mathbb{R}^N$ ($N\geq 2$).
{In a subregion $D\Subset\Omega$, the medium is supposed to be lossy and have a large mass density. We study the asymptotic development of the wave field as the mass density $\rho\rightarrow +\infty$}
and show that the wave field inside $D$ will decay exponentially while the wave filed outside the medium will converge to the one corresponding to a sound-hard obstacle $D\Subset\Omega$ buried in the medium supported in $\Omega\backslash\overline{D}$. Moreover, the normal velocity of the wave field on $\partial D$ from outside $D$ is shown to be vanishing as $\rho\rightarrow +\infty$.
{We derive very accurate estimates for the wave field inside and outside $D$ and on $\partial D$ in terms of $\rho$, and
show that the asymptotic estimates are sharp.
The implication of the obtained results is given for an inverse scattering problem
of reconstructing a complex scatterer.}
\end{abstract}

\pagestyle{myheadings}
\thispagestyle{plain}
\markboth{Hongyu Liu, Zaijiu Shang, Hongpeng Sun and Jun Zou}{Singular perturbation of the reduced wave equation}

\section{Introduction}\label{intro}

{
We shall be concerned in this paper with the following scalar wave equation (see, e.g., \cite{Uno}):
\begin{equation}\label{eq:wave}
\frac{1}{c^2(x)}\frac{\partial^2 U(x,t)}{\partial t^2}+\sigma(x)\frac{\partial U(x,t)}{\partial t}-\nabla\cdot\left(\frac{1}{\rho(x)}\nabla U(x,t)\right)=-F(x,t)
\end{equation}
for all $x\in\mathbb{R}^N$\ $(N\geq 2)$ and $t\in\mathbb{R}_+$. In equation (\ref{eq:wave}),
$U(x,t)$ is the wave field, $c(x)$, $\sigma(x)$ and $\rho(x)$ are positive scalar functions and represent the wave velocity, the damping coefficient and the mass density of the medium respectively. It is supposed that the medium is compactly supported in a bounded domain $\Omega$ in $\mathbb{R}^N$.
We consider the medium outside $\Omega$ to be homogeneous  and no damping present,
so we may assume after normalization that
$c=\tilde{c}_{0}$, $\rho=1$ and $\sigma=0$
in $\Omega^{c} := \mathbb{R}^N\backslash \overline{\Omega}$. Let $D\Subset\Omega$ be a subregion of $\Omega$ and the material parameters inside $D$ be given by
\begin{equation}\label{eq:parameters}
c(x)=c_{0},\ \ \sigma(x)=\sigma_0,\ \ \rho(x)=\varepsilon^{-1} \quad \mbox{for} ~~x\in D\,,
\end{equation}
where $c_0, \sigma_0$ and $\varepsilon$ are positive constants.
}
This work shall be devoted to the study of the asymptotic development of the wave field $U(x,t)$ as
the mass density $\rho$ inside $D$ tends to infinity, i.e., the parameter $\varepsilon\rightarrow 0^+$.
 {We shall consider  the time-harmonic wave propagation, namely to seek a solution
 of (\ref{eq:wave}) in the following form}
\[
U(x,t)= \Re \{u(x)e^{-i\omega t}\},\quad
F(x,t) = \Re \{ f(x)e^{-i \omega t}\},
\]
where $\omega\in\mathbb{R}_+$ is the frequency. By our earlier assumption on the homogeneous space outside the medium $\Omega$, we see the wave number $k=\omega/\tilde{c}_{0}$. We suppose that $f(x)$ is compactly supported outside the inhomogeneous medium, namely $supp(f)\subset B_{R_0}\backslash\overline{\Omega}$ for some $R_0>0$,
where and in the sequel $B_r$
denotes
{
a ball of radius $r$ centered at the origin in $\mathbb{R}^N$.
Factorizing out the time-dependent part, the wave equation (\ref{eq:wave}) reduces to
the following time-harmonic equation:}
\begin{equation}\label{eq:reduced wave}
\nabla\cdot\left(\frac{1}{\rho}\nabla u\right)+k^2\left(\frac{\tilde{c}_{0}^2}{c^2}+i\frac{\sigma \tilde{c}_{0}}{k}\right)u=f(x)\quad\mbox{in\ \ $\mathbb{R}^N$}.
\end{equation}

We shall seek the total wave field of (\ref{eq:reduced wave}) admitting the following asymptotic development as $|x|\rightarrow\infty$:
\begin{equation}\label{eq:farfield}
u(x)=e^{i k x\cdot d}+\frac{e^{ik|x|}}{|x|^{(N-1)/2}}\left\{\mathcal{A}\left(\hat{x}, d,k\right)+\mathcal{O}\left(\frac{1}{|x|}\right)\right\},
\end{equation}
{
where $e^{i k x\cdot d}$ is the incident field, and
$\mathcal{A}(\hat{x},d,k)$ with $\hat{x}=x/|x|$ is known as the scattering amplitude (cf. \cite{ColKre}
\cite{Isa}), with $d\in \mathbb{S}^{N-1}$. For notational convenience, we set
\[
\gamma=\rho^{-1},\ \ q=\frac{\tilde{c}_{0}^2}{c^{2}}+i\frac{\sigma \tilde{c}_{0}}{k}
\quad \mbox{in} ~~\Omega\backslash\overline{D}; \ \  \eta_0=\frac{\tilde{c}_{0}^2}{c_{0}^{2}},\ \ \tau_0=\frac{\sigma_{0}\tilde{c}_{0}}{k} \quad \mbox{in} ~~D\,,
\]
}
and $u^s(x)=u(x)-u^i(x)$ is the scattered field outside the medium region $\Omega$.

{Throughout the rest of the paper, we assume that $\Omega$ and $D$ are both bounded $C^{2}$ domains such that $\mathbb{R}^N\backslash \overline{\Omega}$ and $\Omega\backslash\overline{D}$ are connected. Let $q\in L^\infty(\Omega\backslash\overline{D})$ and $\gamma(x)\in C^{2}(\overline{\Omega}\backslash D)$ satisfying the following physically meaningful conditions:
{
\[
\gamma_0\leq \gamma(x)\leq \Upsilon_0, \quad \Re q(x) \geq \Gamma_{0}, \quad \Im q(x) \geq 0
\quad ~~\mbox{for} ~~x\in\Omega\backslash\overline{D}\,,
\]
}
where $\gamma_0$, $\Upsilon_0$, $\Gamma_{0}$ are positive constants. With all these preparations,
we can formulate our interested problem of finding the total wave field $u(x)$ of form (\ref{eq:farfield}) to the system (\ref{eq:reduced wave}) as follows:
Find $u_\varepsilon  \in H_{loc}^1(\mathbb{R}^N)$ such that
\begin{equation}\label{eq:transmission problem}
\begin{cases}
\displaystyle{\nabla\cdot(\varepsilon\nabla u_\varepsilon)+ k ^2(\eta_0+i\tau_0)u_\varepsilon=0}\quad & \mbox{in \ $D$},\\
\displaystyle{\nabla\cdot(\gamma(x)\nabla u_\varepsilon)+k^2 q(x) u_\varepsilon=0}\quad & \mbox{in \ $\Omega\backslash\overline{D}$},\\
\displaystyle{ \Delta u_\varepsilon^s +k^2  u_\varepsilon^s =f }\quad & \mbox{in \ $\mathbb{R}^N\backslash\overline{\Omega}$},\\
\displaystyle{u_\varepsilon=u^i+u_\varepsilon^s} \quad   & \mbox{in \ $\mathbb{R}^N\backslash\overline{\Omega}$},\\
\ \displaystyle{u_\varepsilon^-=u_\varepsilon^+,\quad \varepsilon\frac{\partial u_\varepsilon^-}{\partial\nu}=\gamma\frac{\partial u_\varepsilon^+}{\partial\nu}}\quad & \mbox{on \ $\partial D$},\\
\ \displaystyle{u_\varepsilon^- = u_\varepsilon^s+u^i}, \quad  \ \displaystyle{\gamma\frac{\partial u_\varepsilon^-}{\partial\nu}= \frac{\partial u_\varepsilon^s}{\partial\nu}+\frac{\partial u^i}{\partial \nu}}\quad & \mbox{on \ $\partial \Omega$},\\
\ \displaystyle{\lim_{|x|\rightarrow \infty}|x|^{(N-1)/2}\left\{\frac{\partial u_\varepsilon^s}{\partial |x|}-ik u_\varepsilon^s \right\}=0},
\end{cases}
\end{equation}
where $\nu$ denotes the exterior unit normal to $\partial D$ or $\partial \Omega$. We use the notations $u_{\varepsilon}^-$,$u_{\varepsilon}^+$ to represent the limits of $u_\varepsilon$ on $\partial D$ or $\partial\Omega$,
taking respectively from
inside and outside $D$ or $\Omega$.}
The last limit in (\ref{eq:transmission problem}) is known as the Sommerfeld radiation condition. The well-posedness of the scattering problem (\ref{eq:transmission problem}) is given in the Appendix and the scattering amplitude in (\ref{eq:farfield}) can be read off from the large asymptotics of $u_\varepsilon^s$.  It is readily seen that $u_\varepsilon$ depends on $\varepsilon$ nonlinearly and so does $u_\varepsilon^s$. In order to present the main results of this paper, we introduce the following scattering problem:

Find $u\in H_{loc}^1(\mathbb{R}^N\backslash\overline{D})$ such that
\begin{equation}\label{eq:soundhard}
\begin{cases}
\ \displaystyle{\nabla\cdot(\gamma(x)\nabla u)+ k^2 q(x) u=0}\quad & \mbox{in \ $\Omega\backslash\overline{D}$},\\
\ \displaystyle{\Delta u^{s} + k^{2}u^{s} = f}\quad &  \mbox{in \ $\mathbb{R}^N\backslash\overline{\Omega}$}, \\
\ \displaystyle{u=u^i+u^s}  \quad & \mbox{in \  $\mathbb{R}^N\backslash\overline{\Omega}$}, \\
\ \displaystyle{\gamma\frac{\partial u^{+}}{\partial\nu}=0} \quad &  \mbox{on \ $\partial D$},\\
\ \displaystyle{u^{-} = u^{s}+u^{i}, \quad \gamma\frac{\partial u^{-}}{\partial\nu}= \frac{\partial u^{s}}{\partial \nu}+\frac{\partial u^i}{\partial \nu}} \quad & \mbox{on \ $\partial \Omega$},\\
\ \displaystyle{\lim_{|x|\rightarrow \infty}|x|^{(N-1)/2}\left\{\frac{\partial u^{s}}{\partial |x|}-i k u^{s} \right\}}=0.
\end{cases}
\end{equation}

{
One can see from (\ref{eq:soundhard}) that the normal velocity of the wave field
vanishes on the boundary $\partial D$,
so the wave can not penetrate inside $D$. }
In the acoustic scattering, $D$ is known as a {\it sound-hard} obstacle, so the system (\ref{eq:soundhard}) is an obstacle scattering
problem with an obstacle buried inside some inhomogeneous medium.
We shall show that the solution $u_\varepsilon$ of
the medium scattering problem (\ref{eq:transmission problem}) will converge to the solution
$u$ of the obstacle scattering problem
(\ref{eq:soundhard}) as $\varepsilon\rightarrow 0^+$,
or the density $\rho$ of the medium $D$ tends to infinity.
This is reflected by the results in the following three theorems, where $C$ and $\widetilde C$ are
generic constants, which depend only on $q, k, \eta_{0}, \tau_{0}, \gamma, \varepsilon_0, D, \Omega, B_{R}$,
but completely independent of $\varepsilon$.
\begin{thm}\label{thm1}
Let $u_\varepsilon\in H_{loc}^1(\mathbb{R}^N)$ and $u\in H_{loc}^1(\mathbb{R}^N\backslash\overline{D})$ be the solutions to (\ref{eq:transmission problem}) and (\ref{eq:soundhard}), respectively.
Then for any $R>R_0$, there exist $\varepsilon_0>0$ and $C>0$ such that
{the following estimate holds for $\varepsilon<\varepsilon_0$:}
\begin{equation}
\|u_\varepsilon -u\|_{H^{1}(B_{R}\backslash\overline{D})} \leq C\varepsilon^{1/2}(\|u^{i}\|_{H^{1}(B_{R}\backslash \overline{\Omega})} + \|f\|_{L^{2}(B_{R_{0}}\backslash \overline{\Omega})})\,.
\end{equation}
%where $C$ depends only on $q, k, \eta_{0}, \tau_{0}, \gamma, D, \Omega, B_{R}$,
%but completely independent of $\varepsilon$.
As a consequence,
{
the scattering amplitude $\mathcal{A}_\varepsilon$  of $u_\varepsilon^s$
converges to the amplitude $\mathcal{A}$ of $u^s$
in the following sense that
\begin{equation}\label{farfield}
\|\mathcal{A}_\varepsilon-\mathcal{A}\|_{C(\mathbb{S}^{N-1})}\leq \widetilde{C} \varepsilon^{1/2}(\|u^{i}\|_{H^{1}(B_{R}\backslash \overline{\Omega})} + \|f\|_{L^{2}(B_{R_{0}}\backslash \overline{\Omega})})
\end{equation}
for some constant $\widetilde{C}>0$ and all $\varepsilon<\varepsilon_0$.}
\end{thm}

The next theorem characterizes the normal velocity of the wave field $u_\varepsilon$ on
the boundary of the medium $D$.

\begin{thm}\label{thm2}
For the solution $u_\varepsilon\in H_{loc}^1(\mathbb{R}^N)$ to the system
(\ref{eq:transmission problem}),
there exists $\varepsilon_0>0$ such that
{
the following estimate holds
for $\varepsilon<\varepsilon_0$:}
\begin{equation}
\left\|\gamma\frac{\partial u_\varepsilon^+}{\partial\nu}\right\|_{H^{-1/2}(\partial D)}\leq C \varepsilon^{1/2}(\|u^{i}\|_{H^{1}(B_{R}\backslash \overline{\Omega})} + \|f\|_{L^{2}(B_{R_{0}}\backslash \overline{\Omega})})\,.
\label{eq:trace1}
\end{equation}
%where $C$ depends only on $q, k, \eta_{0}, \tau_{0}, \gamma, D, \Omega, B_{R}$, but completely independent of $\varepsilon$.
\end{thm}

Moreover, the next lemma indicates that the solution $u_\varepsilon$ inside the medium $D$ decays exponentially.

\begin{thm}\label{thm3}
Let $D_0$ be a subdomain such that $D_{0} \Subset D$
with $\mbox{\em dist}(\partial D_{0}, \partial D)$ $\geq \delta_0>0$,
and $\sqrt{\eta_{0}+i\tau_{0}} = a + bi$ with $a >0, b>0$.
Then for the solution $u_\varepsilon\in H_{loc}^1(\mathbb{R}^N)$
to the system (\ref{eq:transmission problem}),
there exists $\varepsilon_0>0$ such that for $\varepsilon<\varepsilon_0$,
\begin{equation}\label{eq:exponential}
\|u_{\varepsilon}\|_{C(D_{0})} \leq  C \exp({-\frac{kb\delta_0}{2 \sqrt{ \varepsilon}}})\, (\|u^{i}\|_{H^{1}(B_{R}\backslash \overline{\Omega})} + \|f\|_{L^{2}(B_{R_{0}}\backslash \overline{\Omega})})\,.
\end{equation}
%where $C$ depends only on $q, k, \eta_{0}, \tau_{0}, \gamma, D, \Omega, B_{R}$, but completely independent of $\varepsilon$.
\end{thm}

\section{Discussions}

We are interested in the scattering from a compactly supported inhomogeneous isotropic medium,
with a subregion occupied by some medium possessing a large density.
Based on our discussions in the previous section, we let
\begin{equation}\label{eq:medium}
\{\Omega\backslash\overline{D};\gamma,q\}\oplus\{D; \varepsilon, \eta_0+i\tau_0\}
\end{equation}
denote the inhomogeneity supported in $\Omega$ in (\ref{eq:transmission problem}),
and
\begin{equation}\label{eq:complex scatterer}
\{\Omega\backslash\overline{D};\gamma,q\}\oplus D
\end{equation}
denote the scatterer in (\ref{eq:soundhard}), where $D$ is known as an impenetrable
{\it sound-hard obstacle} in the acoustic scattering (cf.\,\cite{ColKre}).
As it can be seen from (\ref{eq:soundhard}),
the wave field for a sound-hard obstacle can not penetrate inside and
the normal wave velocity vanishes on the exterior boundary of the obstacle. We call the scatterer in (\ref{eq:complex scatterer}), composed of an obstacle and a surrounding inhomogeneous medium as a {\it complex scatterer}. In this work, we actually show that
\begin{equation}\label{eq:convergence}
\{\Omega\backslash\overline{D};\gamma,q\}\oplus\{D; \varepsilon, \eta_0+i\tau_0\}\rightarrow \{\Omega\backslash\overline{D};\gamma,q\}\oplus D\quad\mbox{as\ \ $\varepsilon\rightarrow 0^+$},
\end{equation}
in the sense of Theorems~\ref{thm1}--\ref{thm3}.
That is, a sound-hard obstacle can be treated as a medium with extreme material property, namely
with a very large mass density.
{
Despite
the nonlinear nature of the convergence (\ref{eq:convergence}), we can still }
derive very accurate estimates
in a general setting. In addition to provide a mathematical characterization of a physically sound-hard obstacle and its asymptotic connection to media with extreme material properties, we would like to note that the results established in this work could have some interesting implication in the inverse scattering problem of reconstructing a complex scatterer. In fact, it can be seen that a complex scatterer could be reconstructed as a medium, and one could locate the embedded obstacle in the reconstruction as the subregion with a large density parameter.

Finally, we make another {practically meaningful} remark on our study.
In (\ref{eq:medium}), the outer inhomogeneous medium $\{\Omega\backslash\overline{D}; \gamma, q\}$
could be anisotropic, for which one could also show the convergence (\ref{eq:convergence})
by modifying our arguments in the subsequent sections.
However, as mentioned earlier, one of our main motivations is from the inverse scattering problem.
If the surrounding medium is anisotropic, one could not uniquely recover a complex scatterer;
actually one may have the invisibility or virtual reshaping phenomena (see, e.g. \cite{Liu} \cite{U2}).
This is why we focus on the isotropic setting in this work.
The extreme medium inside $D$ is assumed to be lossy, which is
a realistic assumption from the practical viewpoint.

The rest of the paper is organized as follows.
{
In Section~\ref{section:4},
we prove the main results of this work, and
demonstrate the sharpness of
our major theoretical estimates
by considering a special case based on series expansions
in Section~\ref{ballcase:check}.}
%which will demonstrate
%the sharpness of our theoretical estimates in Section~\ref{intro}.

\section{Proofs of the main theorems}\label{section:4}
This section is devoted to the proofs of Theorems~\ref{thm1}--\ref{thm3}
in Section~\ref{intro}. For the purpose we need the following lemma.
\begin{lem}\label{important:lemma}
Consider the following transmission problem
\begin{equation}\label{isotropic1:r}
\begin{cases}
\ \displaystyle{\nabla \cdot(\gamma(x)\nabla v)+ k^2 q(x) v = 0}  & \mbox{in\ ~~$\Omega \backslash \overline{D}$} , \\
\ \displaystyle{\Delta u^{s} + k^2 u^{s} = f}  & \mbox{in\ ~~$\mathbb{R}^N\backslash \overline{\Omega}$}, \\
\displaystyle{\gamma \frac{\partial v}{\partial \nu} = p\in H^{-1/2}(\partial D)} & \mbox{on\ ~~$\partial D$}, \\
v-u^{s}=g_{1}\in H^{1/2}(\partial \Omega) & \mbox{on\ ~~$\partial\Omega$}, \\
\displaystyle{\gamma\frac{\partial v}{\partial \nu} - \frac{\partial u^{s}}{\partial \nu}=g_{2} \in H^{-1/2}(\partial \Omega)} & \mbox{on\ ~~$\partial\Omega$}, \\
\displaystyle{\lim_{|x|\rightarrow +\infty}|x|^{(N-1)/2}\left\{\frac{\partial u^{s}}{\partial |x|}-ik u^{s}\right\}=0.}
\end{cases}
\end{equation}
There exists a unique solution $(v, u^s)\in H^{1}(\Omega\backslash\overline{D})\times H_{loc}^1(\mathbb{R}^N\backslash\overline{\Omega})$ to (\ref{isotropic1:r}), and the solution satisfies
\begin{equation}\label{eq:est1}
\begin{split}
&\|v\|_{H^{1}(\Omega \backslash \overline{D})}+\|u^{s}\|_{H^{1}(B_{R}\backslash \overline{\Omega})}\\
\leq & C(\|p\|_{H^{-1/2}(\partial D)} + \|g_{1}\|_{H^{1/2}(\partial \Omega)} + \|g_{2}\|_{H^{-1/2}(\partial \Omega)} + \|f\|_{L^{2}(B_{R_0}\backslash \overline{\Omega})}),
\end{split}
\end{equation}
where the positive constant $C$ depends only on $\gamma, q, k, \Omega$, $D$ and $B_R$, but independent of $p$, $g_{1}$, $g_{2}$, $f$.
\end{lem}

We could not find some references on the well-posedness of the transmission problem (\ref{isotropic1:r}),
so provide a proof by using a variational technique presented in \cite{FD} and \cite{PH}.
We first demonstrate the following auxiliary lemma.
\begin{lem}\label{lemma:assit}
The system \eqref{isotropic1:r} is uniquely solvable and
it is equivalent to the following truncated system:
find $(v_1,u_1)\in H^1(\Omega\backslash\overline{D})
\times H^1(B_R\backslash\overline{\Omega})$ such that
\begin{equation}\label{dtn:bound}
\begin{cases}
\ \displaystyle{\nabla \cdot (\gamma(x) \nabla v_{1})+ k^2 q(x) v_{1} = 0} \quad & \mbox{in\ ~~$\Omega \backslash \overline{D}$}, \\
\ \displaystyle{\Delta u_{1} + k^2 u_{1} = f} \quad & \mbox{in\ ~~$B_{R}\backslash \overline{\Omega}$}, \\
\gamma\frac{\partial v}{\partial \nu} = p \quad & \mbox{on ~~$\partial D$}, \\
v_{1} -  u_{1}  = g_{1} \quad & \mbox{on ~~$\partial \Omega$}, \\
\gamma \frac{\partial v_{1}}{\partial \nu} - \frac{\partial u_{1}}{\partial \nu} = g_{2} \quad & \mbox{on\ ~~$\partial \Omega$}, \\
\frac{\partial u_{1}}{\partial \nu} = \Lambda u_{1} \quad & \mbox{on\ ~~$\partial B_{R}$},
\end{cases}
\end{equation}
{
where $\Lambda: H^{1/2}(\partial B_{R})\rightarrow H^{-1/2}(\partial B_{R})$
is the Dirichlet-to-Neumann map
%from $H^{1/2}(\partial B_{R})$ to  $H^{-1/2}(\partial B_{R})$
defined by
$\Lambda \psi = \frac{\partial W}{\partial \nu}|_{\partial B_{R}}$
(cf. \cite{FD} \cite{Kir} \cite{PH}),
%\begin{equation}\label{T:operator}
%\Lambda \psi = \frac{\partial W}{\partial \nu}\bigg|_{\partial B_{R}}
%\end{equation}
with $W\in H^{1}_{loc}(\mathbb{R}^N\backslash \overline{B}_{R})$ being the unique solution
to the system
}
\begin{equation}
\begin{cases}
\Delta W + k^2 W = 0 \quad & \mbox{in \ $\mathbb{R}^N \backslash \overline{B}_{R}$}, \\
W  = \psi \in H^{1/2}(\partial B_{R}) \quad & \mbox{on \ $\partial B_R$}, \\
\displaystyle{\lim_{|x|\rightarrow +\infty}|x|^{(N-1)/2}\left\{\frac{\partial W}{\partial |x|}-ik W\right\}=0.}
\end{cases}
\end{equation}
\end{lem}

\begin{proof}
{We first show the uniqueness of the solution
$(v, u^s)$
%\in H^{1}(\Omega\backslash\overline{D})\times H_{loc}^1(\mathbb{R}^N\backslash\overline{\Omega})$
to system \eqref{isotropic1:r}.}
For the purpose
we set $p$, $g_{1}$, $g_{2}$, $f$ to be all zeros. Multiplying the first and second equations of \eqref{isotropic1:r}, respectively, by $\bar{v}$ and $\bar{u}^{s}$, and integrating by parts in $\Omega \backslash \overline{D}$ and $B_{R}\backslash \overline{\Omega}$, together with the use of the boundary conditions on $\partial D$ and $\partial \Omega$, we have
\begin{equation}\label{unique1}
\begin{split}
&-\int_{\Omega \backslash \overline{D}} \gamma|\nabla v|^2 dx + \int_{\Omega \backslash \overline{D}}k^2 q |v|^2 dx - \int_{B_{R}\backslash \overline{\Omega}}|\nabla u^{s}|^2 dx\\
& + \int_{B_{R}\backslash \overline{\Omega}} k^2 |u^{2}|^2 dx + \int_{\partial B_{R}}\frac{\partial u^{s}}{\partial \nu}\bar{u}^{s}ds = 0.
\end{split}
\end{equation}
{
Taking the imaginary part of both sides of \eqref{unique1}, we derive
}
\[
\Im \int_{\partial B_{R}}\frac{\partial u^{s}}{\partial \nu}\bar{u}^{s}ds = - \Im \int_{\Omega \backslash \overline{D}}k^2 q |v|^2 dx \leq 0.
\]
{Then by Rellich's lemma (cf. \cite{ColKre}) we know $u^{s}$ is zero outside $B_{R}$,
which with
the unique continuation implies that} $u^s=0$ in $\Omega\backslash\overline{D}$ and $v=0$ in $D$.

Next we show the equivalence between systems \eqref{isotropic1:r} and
\eqref{dtn:bound}. By the definition of $\Lambda$, {we see that if $(v,u^{s})$ {solves the system} \eqref{isotropic1:r}, then $(v_{1}=v, u_{1}=u^{s}|_{B_{R}\backslash \overline{\Omega}})$ is the solution {to the system \eqref{dtn:bound}}. On the other hand, by applying
the Green's representation
(cf.~\cite{ColKre}(2.4)) to
the solution $(v_{1}, u_{1})$ of \eqref{dtn:bound} we obtain that }
\begin{equation}
\begin{split}
& u_{1}(x) = - \int_{\partial \Omega}\left(\frac{\partial  u_{1}(y)}{\partial \nu(y)}\Phi(x,y) - u_{1}(y)\frac{\partial \Phi(x,y)}{\partial \nu(y)}\right)ds(y)\\
+ & \int_{\partial B_{R}}\left(\Lambda u_{1}(y)\Phi(x,y) - u_{1}(y)\frac{\partial \Phi(x,y)}{\partial \nu(y)}\right) ds(y) -\int_{B_{R}\backslash \overline{\Omega}}f(y)\Phi(x,y)dy,
\end{split}
\end{equation}
for $x\in B_R\backslash\overline{\Omega}$, where
\begin{equation}\label{eq:fun}
\Phi(x,y)=\frac{i}{4}\left(\frac{k}{2\pi|x-y|}\right)^{(N-2)/2}H_{(N-2)/2}^{(1)}(k|x-y|)
\end{equation}
is the outgoing Green's function. By definition of $\Lambda$ and the radiation of $\Phi(x,y)$
{(cf. pp. 98 in \cite{FD}, and \cite{PH})}
\[
\int_{\partial B_{R}}\left(\Lambda u_{1}(y)\Phi(x,y) - u_{1}(y)\frac{\partial \Phi(x,y)}{\partial \nu(y)}\right) ds(y) = 0.
\]
Hence,
\begin{equation}\label{eq:extension}
u_{1}(x) = - \int_{\partial \Omega}\left(\frac{\partial  u_{1}(y)}{\partial \nu(y)}\Phi(x,y) - u_{1}(y)\frac{\partial \Phi(x,y)}{\partial \nu(y)}\right)ds(y) -\int_{B_{R}\backslash \overline{\Omega}}f(y)\Phi(x,y)dy.
\end{equation}
It is clear that $u_1$ can be readily extended to an $H^1_{loc}(\mathbb{R}^N\backslash\overline{\Omega})$ function, which we still denote by $u_1$.  {We can see that} $u_1$
satisfies the Sommerfeld radiation condition, which together with the uniqueness of solution to (\ref{isotropic1:r}) implies that $u_1=u^s$.
\end{proof}

With the uniqueness and equivalence in Lemma \ref{lemma:assit},  we can apply the variational technique to study the reduced problem \eqref{dtn:bound} to prove Lemma \ref{important:lemma}.
\begin{proof}
[Proof of Lemma~\ref{important:lemma}] Without of loss generality, we assume $k^{2}$ is not a Dirichlet eigenvalue of $-\Delta$ in $B_{R}\backslash \overline{\Omega}$, and introduce the following auxiliary system
\begin{equation}\label{shuzhu:v}
\begin{cases}
-\Delta\tilde{v}-k^{2}\tilde{v} = 0 \quad  & \mbox{in \  $B_{R}\backslash \overline{\Omega}$}, \\
\tilde{v} = g_{1}\quad & \mbox{on \ $\partial\Omega$}, \\
\tilde{v} = 0 \quad & \mbox{on \ $\partial B_R$}.
\end{cases}
\end{equation}
It is easy to see
$\|\tilde{v}\|_{H^{1}(B_{R}\backslash \overline{\Omega})} \leq C\|g_{1}\|_{H^{1/2}(\partial \Omega)}$.
We now set
\begin{equation}
w(x):=\begin{cases}
v_{1}(x), \quad & x \in \Omega \backslash \overline{D}, \\
u_{1}(x) + \tilde{v}(x), \quad & x \in B_{R}\backslash \overline{\Omega}.
\end{cases}
\end{equation}
{We can check that $w \in H^{1}(B_{R})$ satisfies the following equation:}
\begin{equation}\label{va:w}
\begin{cases}
\nabla \cdot(\gamma(x)\nabla w)+ k^2 q(x) w = 0 \quad & \mbox{in\ \ $\Omega \backslash \overline{D}$} , \\
\Delta w + k^2 w = f \quad & \mbox{in\ \ $B_{R}\backslash \overline{\Omega}$}, \\
\gamma\frac{\partial w}{\partial \nu} = p  \quad & \mbox{on\ \ $\partial D$},\\
w^{-} = w^{+} \quad & \mbox{on\ \ $\partial \Omega$}, \\
\gamma\frac{\partial w^{-}}{\partial \nu} = \frac{\partial w^{+}}{\partial \nu} + g_{2} - \frac{\partial \tilde{v}}{\partial \nu} \quad & \mbox{on\ \ $ \partial \Omega$}, \\
\frac{\partial w}{\partial \nu} = \Lambda w + \frac{\partial \tilde{v}}{\partial \nu} \quad & \mbox{on\ \ $\partial B_{R}$}.
\end{cases}
\end{equation}
Next, we define $\Lambda_{0}$:
$H^{1/2}(\partial B_{R}) \rightarrow H^{-1/2}(\partial B_{R})$ by
\[
\Lambda_{0}\psi_{1} = \frac{\partial W_{1} }{\partial \nu}\bigg|_{\partial B_{R}},
\]
where $W_{1}\in H_{loc}^1(\mathbb{R}^N\backslash\overline{B}_R)$ is the unique solution of
the system:
\begin{equation}
\begin{cases}
-\Delta W_{1} = 0 \quad & \mbox{in\ \ $\mathbb{R}^N \backslash \overline{B}_{R}$},\\
W_{1} = \psi_{1}\in H^{1/2}(\partial B_{R})\quad & \mbox{on\ \ $\partial B_R$,}
\end{cases}
\end{equation}
{and satisfies} the decay property at infinity, namely
$W_{1} = \mathcal{O}({|x|^{-1}})$ for $N = 3$, and
$W_{1} = \mathcal{O}(\log|x|)$ for $N = 2$, as $|x| \rightarrow +\infty$.

It is known that (cf. \cite{FD} and \cite{PH})
\begin{equation}\label{T0}
-\int_{\partial B_{R}}\bar{\psi}_{1}\Lambda_{0}\psi_{1}ds \geq 0, \quad   \forall \psi_{1} \in H^{1/2}(\partial B_{R}),
\end{equation}
and $\Lambda-\Lambda_{0}$ is compact from $H^{1/2}(\partial B_{R})$ to $H^{-1/2}(\partial B_{R})$.
Then for any $\varphi \in H^{1}(B_{R})$, using the test function $\bar{\varphi}$ we can easily
derive the variational formulation of system \eqref{va:w}: ~find $w\in H^{1}(B_{R})$ such that
\begin{equation}\label{eq:weak}
a_{1}(w,\varphi) +a_{2}(w,\varphi) = {\cal F}(\varphi)
\end{equation}
where the bilinear forms $a_{1}$ and $a_{2}$ and the linear functional
${\cal F}$ are given by
%%
%\begin{equation}
%\begin{split}
%&-\int_{\partial D}p \bar{\varphi}ds + \int_{\partial \Omega}(g_{2}-\frac{\partial \tilde{v}}{\partial \nu})\bar{\varphi} ds + \int_{\partial B_{R}}\frac{\partial \tilde{v}}{\partial \nu}\bar{\varphi}ds -\int_{B_{R}}f\bar{\varphi}dy\\
%=& \int_{\Omega\backslash \overline{D}}\gamma\nabla w \cdot \nabla \bar{\varphi} dy - \int_{\Omega\backslash \overline{D}} k^2 q w \bar{\varphi}dy + \int_{B_{R}\backslash \overline{\Omega}}\nabla w \cdot \nabla \bar{\varphi} dy\\
%&- \int_{B_{R}\backslash \overline{\Omega}}k^2 w \bar{\varphi}dy - \int_{\partial B_{R}}\Lambda w \bar{\varphi}ds.
%\end{split}
%\end{equation}
%For simplicity, we now introduce
\begin{align}
a_{1}(w,\varphi) :=& \int_{\Omega\backslash \overline{D}}\gamma\nabla w \cdot \nabla \bar{\varphi} dy + \int_{\Omega\backslash \overline{D}} k^2w \bar{\varphi}dy + \int_{B_{R}\backslash \overline{\Omega}}\nabla w \cdot \nabla \bar{\varphi} dy
\nonumber\\
& + \int_{B_{R}\backslash \overline{\Omega}}k^2 w \bar{\varphi}dy - \int_{\partial B_{R}}\Lambda_{0}w \bar{\varphi}ds,\\
a_{2}(w,\varphi) :=& - \int_{\Omega\backslash \overline{D}} k^2(q+1)w \bar{\varphi}dy - 2\int_{B_{R}\backslash \overline{\Omega}}k^2 w \bar{\varphi}dy -  \int_{\partial B_{R}}(\Lambda-\Lambda_{0})w \bar{\varphi}ds,\\
\mathscr{F}(\varphi) :=& -\int_{\partial D}p \bar{\varphi}ds + \int_{\partial \Omega}(g_{2}-\frac{\partial \tilde{v}}{\partial \nu})\bar{\varphi} ds + \int_{\partial B_{R}}\frac{\partial \tilde{v}}{\partial \nu}\bar{\varphi}ds -\int_{B_{R}}f\bar{\varphi}dy.\label{F:right}
\end{align}
%and
%\begin{equation}\label{F:right}
%\mathscr{F}(\varphi) := -\int_{\partial D}p \bar{\varphi}ds + \int_{\partial \Omega}(g_{2}-\frac{\partial \tilde{v}}{\partial \nu})\bar{\varphi} ds + \int_{\partial B_{R}}\frac{\partial \tilde{v}}{\partial \nu}\bar{\varphi}ds -\int_{B_{R}}f\bar{\varphi}dy.
%\end{equation}
Using \eqref{T0} we can readily
verify that for any $\phi, \varphi \in H^{1}(B_{R})$,
\begin{equation}
|a_{1}(\phi, \varphi)| \leq C_{1}\|\phi\|_{H^{1}(B_{R})}\|\varphi\|_{H^{1}(B_{R})} \quad\mbox{and}\quad
a_{1}(\varphi,\varphi) \geq C_{2}\|\varphi\|_{H^{1}(B_{R})}^2
\end{equation}
for some constants $C_{1}$ and $C_{2}$.
Then by Lax-Milgram lemma there exists a {bounded} operator $\mathcal{L} : H^{1}(B_{R}) \rightarrow H^{1}(B_{R})$ such that
\begin{equation}
a_{1}(w,\varphi) = (\mathcal{L}w ,\varphi), \quad \forall \varphi, w \in H^{1}(B_{R}),
\end{equation}
where and in the following, $(\cdot,\cdot)$ denotes the inner product in $H^1(B_R)$.
Moreover, the inverse $\mathcal{L}^{-1}$ exists and is bounded.
By Riesz representation theorem, we also know that there exist bounded operators $\mathcal{K}_{1}, \mathcal{K}_{2}:H^{1}(B_{R})\rightarrow H^{1}(B_{R})$ such that
\begin{equation}\label{a3:func}
 a_{3}(w,\varphi): = \int_{\Omega\backslash \overline{D}} k^2(q+1)w \bar{\varphi}dy + 2\int_{B_{R}\backslash \overline{\Omega}}k^2 w \bar{\varphi}dy = (\mathcal{K}_{1}w,\varphi)
\end{equation}
and
\begin{equation}\label{a4:func}
a_{4}(w,\varphi): = \int_{\partial B_{R}}(\Lambda-\Lambda_{0})w \bar{\varphi}ds = (\mathcal{K}_{2}w, \varphi).
\end{equation}
We now claim that both $\mathcal{K}_1$ and $\mathcal{K}_2$ are compact.
In fact, let $\{w_{n}\}_{n \in \mathbb{N}}$ be a bounded sequence in $H^{1}(B_{R})$ and $\|w_{n}\|_{H^{1}(B_{R})} \leq M$, and we can assume that
$w_{n} \rightharpoonup w_{0}$ in $H^{1}(B_{R})$. Since $H^{1}(B_{R})\hookrightarrow L^{2}(B_{R})$ is compact, we know $w_{n} \rightarrow w_{0}$ in $L^{2}(B_{R})$.
By \eqref{a3:func} {we can write}
\begin{equation}\label{eq:a111}
a_{3}(w_{n}-w_{0}, \varphi) = (\mathcal{K}_{1}(w_{n}-w_{0}),\varphi).
\end{equation}
Taking $\varphi= \mathcal{K}_1(w_n-w_0)$ and using (\ref{a3:func}),
we can verify that
\[
\|\mathcal{K}_{1}(w_{n}-w_{0})\|_{H^{1}(B_{R})} \leq 4Mk^2\max\{\||q+1\|_{L^\infty(\Omega\backslash\overline{D})}, 2\}  \|\mathcal{K}_{1}\| \|w_{n}-w_{0}\|_{L^{2}(B_{R})} \rightarrow 0,
\]
which implies the compactness of $\mathcal{K}_{1}$. In a similar manner, we can prove the compactness of $\mathcal{K}_2$. Indeed, let $w_{n} \rightharpoonup w_{0}$ in $H^{1}(B_{R})$, and by trace theorem, $w_{n}|_{\partial B_{R}} \rightharpoonup w_{0}|_{\partial B_{R}}$ in $H^{1/2}(\partial B_{R})$. Since $\Lambda-\Lambda_{0}: H^{1/2}(\partial B_{R}) \rightarrow H^{-1/2}(\partial B_{R})$ is compact, we see $(\Lambda-\Lambda_{0})w_{n}\rightarrow(\Lambda-\Lambda_{0})w_{0}$ in $H^{-1/2}(\partial B_{R})$.
By \eqref{a4:func} we can write
\[
a_{4}(w_{n}-w_{0}, \varphi) = (\mathcal{K}_{2}(w_{n}-w_{0}),\varphi).
\]
Taking $\varphi= \mathcal{K}_{2}(w_{n}-w_{0})$ and using (\ref{a4:func}), one has
\[
\begin{split}
\|\mathcal{K}_{2}(w_{n}-w_{0})\|_{H^{1}(B_{R})} &\leq \|(\Lambda-\Lambda_{0})(w_{n}-w_{0})\|_{H^{-1/2}(\partial B_{R})}\|\mathcal{K}_{2}(w_{n}-w_{0})\|_{H^{1/2}(\partial B_{R})}\\
&\leq C_{3} M \|(\Lambda-\Lambda_{0})(w_{n}-w_{0})\|_{H^{-1/2}(\partial B_{R})} \|\mathcal{K}_{2}\| \rightarrow 0,
\end{split}
\]
which implies the compactness of $\mathcal{K}_{2}$.

Since $\mathcal{L}$ is bounded and invertible, and $\mathcal{K}_{1} + \mathcal{K}_{2}$ is compact, we know $\mathcal{L}-(\mathcal{K}_1+\mathcal{K}_2)$ is a Fredholm operator of index zero. By the uniqueness of (\ref{isotropic1:r}), $(\mathcal{L}-(\mathcal{K}_1+\mathcal{K}_2))^{-1}$ is bounded. On the other hand, it is straightforward to show
{\small
\[
|F(\varphi)| \leq C(\|p\|_{H^{-1/2}(\partial D)}+\|g_{1}\|_{H^{1/2}(\partial \Omega)}
+ \|g_{2}\|_{H^{-1/2}(\partial \Omega)} + \|f\|_{L^{2}(B_{R_{0}+1}\backslash B_{R_0})})\|\varphi\|_{H^{1}(B_{R})},
\]
}
which readily implies (\ref{eq:est1}).
\end{proof}

The next lemma presents some important a priori estimates of the solution $u_\varepsilon$
to (\ref{eq:transmission problem}) in terms of $\varepsilon$.
\begin{lem}\label{mainlem}
Let $u_\varepsilon\in H_{loc}^1(\mathbb{R}^N)$ be the unique solution to (\ref{eq:transmission problem}). There exists $\varepsilon_0>0$ such that {the following estimates hold for all $\varepsilon<\varepsilon_0$,
\begin{eqnarray}
\|u_{\varepsilon}\|_{H^{1}(B_{R}\backslash \overline{D})} &\leq& C_{1} (\|f\|_{L^{2}(B_{R_{0}} \backslash \overline{\Omega})} +\|u^{i}\|_{H^{1}(B_{R}\backslash \overline {\Omega})})\,, \label{uniform:wellpose1}\\
%\end{equation}
%and
%\begin{equation}
\sqrt{\varepsilon}\|u_{\varepsilon}\|_{H^{1}(D)} &\leq& C_{2} (\|f\|_{L^{2}(B_{R_{0}} \backslash \overline{\Omega})}+ \|u^{i}\|_{H^{1}(B_{R}\backslash \overline {\Omega})})\label{uniform:wellpose2}
\end{eqnarray}
}
where the constants $C_1$ and $C_2$ are independent of $\varepsilon$.
\end{lem}
\begin{proof}
%By the well-posedness of (\ref{eq:transmission problem}) (see Appendix),
%we know that (\ref{uniform:wellpose1}) and (\ref{uniform:wellpose2}) hold for each fixed $\varepsilon>0$.
%It suffices for us to prove that $C_1$ and $C_2$ are independent of $\varepsilon$
%when $\varepsilon$ is sufficiently small.
%
Multiplying $\bar{u}_\varepsilon$ to the both sides of the first and second equations of \eqref{eq:transmission problem} and integrating over $\Omega$, we have
\begin{equation}\label{inter:omega}
\begin{split}
&-\int_{D}\varepsilon |\nabla u_\varepsilon|^2 dy + \int_{D}k^2(\eta_{0}+ i \tau_{0})|u_\varepsilon|^2 dy
-\int_{\Omega \backslash \overline{D}}\gamma |\nabla u_\varepsilon|^2 dy\\
&+ \int_{\Omega \backslash \overline{D}} k^2 q|u_\varepsilon|^2 dy + \int_{\partial \Omega}\gamma \frac{\partial u_\varepsilon}{\partial \nu}\bar{u}_\varepsilon ds = 0.
\end{split}
\end{equation}
Then multiplying $\bar{u}_\varepsilon^s$ to the both sides of the third equation of \eqref{eq:transmission problem} and integrating over $B_{R}\backslash \overline{\Omega}$, we obtain
\begin{equation}\label{inter:br}
\begin{split}
&-\int_{\partial \Omega}\frac{\partial u_\varepsilon^s}{\partial \nu}\bar{u}_\varepsilon^s ds + \int_{\partial B_{R}}\frac{\partial u_\varepsilon^s}{\partial \nu}\bar{u}_\varepsilon^s ds -\int_{B_{R}\backslash \overline{\Omega}}|\nabla u_\varepsilon^s|^2dy\\
&+\int_{B_{R}\backslash \overline{\Omega}}k^2|u_\varepsilon^s|^2dy   = \int_{B_{R}\backslash \overline{\Omega}} f\bar{u}_\varepsilon^sdy.
\end{split}
\end{equation}
By adding up \eqref{inter:omega} and \eqref{inter:br},
using the corresponding transmission conditions and then
taking the imaginary and real parts of the resulting equation, we derive
\begin{equation}\label{image}
\begin{split}
&\int_{D}k^2\tau_{0}|u_\varepsilon|^2 dy + \int_{\Omega \backslash \overline{D}} k^2 \Im q|u_\varepsilon|^2 dy + \Im\int_{\partial \Omega}\frac{\partial u_\varepsilon^s}{\partial \nu}\bar{u}^i ds + \Im\int_{\partial \Omega}\frac{\partial u^i}{\partial \nu}\bar{u}_\varepsilon^s ds\\
&+  \Im\int_{\partial \Omega}\frac{\partial u^i}{\partial \nu}\bar{u}^i ds + \Im \int_{\partial B_{R}}\frac{\partial u_\varepsilon^s}{\partial \nu}\bar{u}_\varepsilon^s ds= \Im \int_{B_{R}\backslash \overline{\Omega}} f\bar{u}_\varepsilon^sdy
\end{split}
\end{equation}
and
\begin{equation}\label{real}
\begin{split}
&-\int_{D}\varepsilon |\nabla u_\varepsilon|^2 dy + \int_{D}k^2\eta_{0}|u_\varepsilon|^2 dy
-\int_{\Omega \backslash \overline{D}}\gamma |\nabla u_\varepsilon|^2 dy\\
&+ \int_{\Omega \backslash \overline{D}} k^2 \Re q|u_\varepsilon|^2 dy + \Re\int_{\partial \Omega}\frac{\partial u_\varepsilon^s}{\partial \nu}\bar{u}^i ds + \Re\int_{\partial \Omega}\frac{\partial u^i}{\partial \nu}\bar{u}_\varepsilon^s ds \\
&+  \Re\int_{\partial \Omega}\frac{\partial u^i}{\partial \nu}\bar{u}^i ds+\Re \int_{\partial B_{R}}\frac{\partial u_\varepsilon^s}{\partial \nu}\bar{u}_\varepsilon^s ds-\int_{B_{R}\backslash \overline{\Omega}}|\nabla u_\varepsilon^s|^2dy\\
&+\int_{B_{R}\backslash \overline{\Omega}}k^2|u_\varepsilon^s|^2dy = \Re \int_{B_{R}\backslash \overline{\Omega}} f\bar{u}_\varepsilon^sdy.
\end{split}
\end{equation}
From \eqref{image}, one has by direct verification that
\begin{equation}\label{L2D}
\begin{split}
\|u_\varepsilon\|_{L^{2}(D)}^2 \leq & \widetilde{C}\bigg(\|u_\varepsilon\|_{L^{2}(\Omega \backslash \overline{D})}^2 + (\|u^i\|_{H^{1}(B_{R}\backslash \overline{\Omega})} + \|u_\varepsilon^s\|_{H^{1}(B_{R}\backslash \overline{\Omega})})^2\\
&+ \|f\|_{L^{2}(B_{R}\backslash \overline{\Omega})}\|u_\varepsilon^s\|_{H^{1}(B_{R}\backslash \overline{\Omega})}\bigg)\\
\leq &\ 8\widetilde{C}\left(\|u_\varepsilon\|_{H^{1}(B_{R}\backslash \overline{D})}^2+\|u^i\|_{H^{1}(B_{R}\backslash \overline{\Omega})}^2 +\|f\|_{L^{2}(B_{R}\backslash \overline{\Omega})}^2 \right),
\end{split}
\end{equation}
where $\widetilde{C}$ depends only on $\eta_{0}, \tau_{0}, k, q, \Omega, B_{R}$.
We can readily check by \eqref{real}  that
\begin{equation}\label{L2gD}
\begin{split}
\int_{D}\varepsilon |\nabla u_\varepsilon|^2 dy & \leq \widetilde{C_{2}}\bigg(\|u_\varepsilon\|_{L^{2}(D)}^2 + \|u_\varepsilon\|_{H^{1}(B_{R}\backslash \overline{D})}^2+\|u^i\|_{H^{1}(B_{R}\backslash \overline{\Omega})}^2\\
&+\|f\|_{L^{2}(B_{R}\backslash \overline{\Omega})}\|u_{\varepsilon}^s\|_{H^{1}(B_{R}\backslash \overline{\Omega})}\bigg),
\end{split}
\end{equation}
where $\widetilde{C_{2}}$ depends only on $k, \eta_{0}, q, \gamma, \Omega, B_{R}$.
Combining \eqref{L2D} and \eqref{L2gD}, we see that there exists a constant $\widetilde{C_{3}}$ dependent only on $k, q, \eta_{0}, \tau_{0}, \gamma, \Omega, B_{R}$, such that for $\varepsilon<1$,
\begin{equation}\label{esti:interout}
\sqrt{\varepsilon}\|u_\varepsilon\|_{H^{1}(D)} \leq \widetilde{C_{3}}
\left( \|u_\varepsilon\|_{H^{1}(B_{R}\backslash \overline{D})}^2+\|u^i\|_{H^{1}(B_{R}\backslash \overline{\Omega})}^2 +\|f\|_{L^{2}(B_{R}\backslash \overline{\Omega})}^2
\right)^{1/2}\,.
\end{equation}

Next, we prove \eqref{uniform:wellpose1} by contradiction. Suppose \eqref{uniform:wellpose1} is not true, then without loss of generality, we can assume that {for each $n\in\mathbb{N}$, there exist $f^{n}$ and $u^{i}_{n}$ such that $\|f^n\|_{L^{2}(B_{R_{0}} \backslash \overline{\Omega})}+ \|u^{i}_{n}\|_{H^{1}(B_{R}\backslash \overline {\Omega})} = 1$ and the corresponding solution $u_{\varepsilon}^n$ tends to infinity,
i.e.,} $\|u_{\varepsilon}^n\|_{H^{1}(B_{R}\backslash \overline{D})} \rightarrow +\infty$ as $\varepsilon\rightarrow 0^+$.
Let
\begin{equation}\label{scaling1}
\begin{split}
&v_{\varepsilon,n} = \frac{u_{\varepsilon}^n}{\|u_{\varepsilon}^n\|_{H^{1}(B_{R}\backslash \overline{D})}}, \quad v_{\varepsilon,n}^{i} = \frac{u^{i}}{\|u_{\varepsilon}^n\|_{H^{1}(B_{R}\backslash \overline{D})}},\\
& f_{\varepsilon}^{n} = \frac{f^{n}}{\|u_{\varepsilon}^{n}\|_{H^{1}(B_{R}\backslash \overline{D})}}, \quad v_{\varepsilon,n}^{s} = \frac{u_{\varepsilon}^{n,s}}{\|u_{\varepsilon}^n\|_{H^{1}(B_{R}\backslash \overline{D})}}.
\end{split}
\end{equation}
Clearly, $v_{\varepsilon,n} \in H^{1}_{loc}(\mathbb{R}^N)$ is the unique solution of \eqref{eq:transmission problem} with the incident wave {$ v_{\varepsilon,n}^{i}$} and the source $f_{\varepsilon}^{n} $. We have
\begin{equation}\label{V:F}
\|v_{\varepsilon,n}\|_{H^{1}(B_{R}\backslash \overline{D})}= 1,\quad \|f_{\varepsilon}^{n}\|_{L^{2}(B_{R}\backslash \overline{\Omega})} \rightarrow 0, \quad \|v_{\varepsilon,n}^{i}\|_{H^{1}(B_{R}\backslash \overline{\Omega})} \rightarrow 0.
\end{equation}
By a completely similar argument as we did in
deriving (\ref{esti:interout}), we can show that for sufficiently large $n$,
\begin{equation}\label{scal:v}
\begin{split}
\sqrt{\varepsilon}\|v_{\varepsilon,n}\|_{H^{1}(D)} \leq & \widetilde{C_{3}} \left({\|v_{\varepsilon,n}\|_{H^{1}(B_{R}\backslash \overline{D})}^2+\|v_{\varepsilon,n}^i\|_{H^{1}(B_{R}\backslash \overline{\Omega})}^2 +\|f_{\varepsilon}^n\|_{L^{2}(B_{R}\backslash \overline{\Omega})}^2}\right)^{1/2}\\
\leq & \widetilde{C_{3}}\sqrt{2}.
\end{split}
\end{equation}
{By taking the trace and using the transmission condition on $\partial D$ and
(\ref{scal:v}), we know the existence of }
a constant $\widetilde{C_{4}}$ depending only on $D$ such that
\begin{equation}\label{partial:D}
\left\|\gamma \frac{\partial v_{\varepsilon,n}^{+}}{\partial \nu}\right\|_{H^{-1/2}(\partial D)} = \left\|\varepsilon\frac{\partial v_{\varepsilon,n}^{-}}{\partial\nu}\right\|_{H^{-1/2}(\partial D)} \leq \widetilde{C_{4}}\widetilde{C_{3}}\sqrt{2}\varepsilon^{1/2}.
\end{equation}
Noting that $(v_{\varepsilon,n}|_{\Omega\backslash\overline{D}},v_{\varepsilon,n}^{s}|_{\mathbb{R}^N\backslash\overline{\Omega}})$ is the unique solution of \eqref{isotropic1:r} with $p = \gamma\frac{\partial v_{\varepsilon,n}^{+}}{\partial \nu}|_{\partial D}$, $g_{1} = v_{\varepsilon,n}^i|_{\partial \Omega}$, $g_{2} = \frac{\partial v_{\varepsilon,n}^i}{\partial \nu}|_{\partial \Omega}$, then by Lemma \ref{important:lemma} we have
\begin{equation}\label{eq:final}
\begin{split}
&\|v_{\varepsilon,n}\|_{H^{1}(B_{R}\backslash \overline{D})} \\
\leq &\ C\left(\left\|\gamma \frac{\partial v_{\varepsilon,n}^{+}}{\partial \nu}\right\|_{H^{-1/2}(\partial D)} + \left\|f_{\varepsilon}^n\right\|_{L^{2}(B_{R}\backslash \overline{\Omega})}+ \left\|v_{\varepsilon,n}^i\right\|_{H^{1}(B_{R}\backslash \overline{\Omega})}\right).
\end{split}
\end{equation}
By \eqref{V:F}, \eqref{partial:D} and (\ref{eq:final}), {we further derive }
\[
\|v_{\varepsilon,n}\|_{H^{1}(B_{R}\backslash \overline{D})} \rightarrow 0\quad\mbox{as\ \ $\varepsilon\rightarrow 0^+$},
\]
{which contradicts with the equality} $\|v_{\varepsilon,n}\|_{H^{1}(B_{R}\backslash \overline{D})}=1$ and thus proves (\ref{uniform:wellpose1}).

Now by combining \eqref{uniform:wellpose1} with \eqref{esti:interout},
we obtain \eqref{uniform:wellpose2}.
%which completes the proof.
\end{proof}

We are now in a position to present the proofs of Theorems~\ref{thm1}--\ref{thm3}.

\begin{proof}
[Proof of Theorem~\ref{thm2}] This is a direct consequence of \eqref{uniform:wellpose2} in Lemma \ref{mainlem}. Indeed, by taking the trace on $\partial D$, we see
\[
\left\|\frac{\partial u_{\varepsilon}^{-}}{\partial \nu}\right\|_{H^{-1/2}(\partial D)} \leq \widetilde{C}\|u_{\varepsilon}\|_{H^{1}(D)},
\]
where $\widetilde{C}$ depends only on $D$.
Then by the transmission condition on $\partial D$, we readily derive (\ref{eq:trace1}):
\begin{equation*}\label{neumann:D}
\begin{split}
\left\|\gamma\frac{\partial u_{\varepsilon}^{+}}{\partial \nu}\right\|_{H^{-1/2}(\partial D)} = \left\|\varepsilon \frac{\partial u_{\varepsilon}^{-}}{\partial \nu}\right\|_{H^{-1/2}(\partial D)}
\leq \ C \varepsilon^{1/2} \left(\|f\|_{L^{2}(B_{R_{0}} \backslash \overline{\Omega})}+ \|u^{i}\|_{H^{1}(B_{R}\backslash \overline {\Omega})}\right).
\end{split}
\end{equation*}
\end{proof}

\begin{proof}
[Proof of Theorem~\ref{thm1}] Let
$
V = u_\varepsilon - u$, $V^s = u_\varepsilon^s - u^s.
$
One can verify directly that $V$ satisfies equation \eqref{isotropic1:r} with $f=0$, $p =\gamma \frac{\partial V}{\partial \nu} = \gamma\frac{\partial u_{\varepsilon}^{+}}{\partial \nu}|_{\partial D}$ and $g_{1}=g_{2} =0$. Then by Lemma \ref{important:lemma} and Theorem \ref{thm2}, we have
\begin{equation}\label{eq:eee}
\begin{split}
& \|u_\varepsilon-u\|_{H^1(B_R\backslash\overline{D})}=\|V\|_{H^{1}(B_{R}\backslash\overline{D})} \leq C \left\|\gamma\frac{\partial u_{\varepsilon}^{+}}{\partial \nu}\right\|_{H^{-1/2}(\partial D)}\\
\leq & \ C\varepsilon^{1/2} \left(\|f\|_{L^{2}(B_{R_{0}} \backslash \overline{\Omega})}+ \|u^{i}\|_{H^{1}(B_{R}\backslash \overline {\Omega})}\right).
\end{split}
\end{equation}
Finally we know from \cite{ColKre} (pp.21) that
\begin{equation}\label{eq:eee2}
\left(\mathcal{A}_\varepsilon-\mathcal{A}\right)(\hat{x})= \zeta\int_{\partial B_{R}}\left\{V^{s}\frac{e^{-ik\hat{x}\cdot y}}{\partial \nu} - \frac{\partial V^{s}}{\partial \nu}e^{-ik\hat{x}\cdot y}\right\}ds(y), \quad \hat{x} \in \mathbb{S}^{N-1}
\end{equation}
where $\zeta=1/4\pi$ for $N=3$ and $\zeta=\frac{e^{i\frac{\pi}{4}}}{\sqrt{8\pi k}}$ for $N=2$.
Using (\ref{eq:eee}) and (\ref{eq:eee2}), one can derive \eqref{farfield} by some
straightforward estimates.
\end{proof}
\begin{proof}
[Proof of Theorem~\ref{thm3}]
We shall make use of the following integral representation of the wave field inside $D$ (cf.\,\cite{ColKre}):
\begin{equation}\label{eq:integral}
u_\varepsilon(x)  = \int_{\partial D} \left\{\frac{\partial u_{\varepsilon}^{-}}{\partial \nu}(y)G(x,y) - u_{\varepsilon}^{-}(y) \frac{\partial G(x,y)}{\partial \nu(y)}\right\}ds(y), \quad x  \in D,
\end{equation}
where $G(x, y)$ is the fundamental solution corresponding to the first equation of
(\ref{eq:transmission problem})  {and is given by}
\begin{equation}
G(x,y)=\frac{e^{i\tilde{k}|x-y|}}{4 \pi |x-y|} \quad \mbox{for} ~~N =3\,;
\quad G(x,y)= \frac{i}{4}H_{0}^{(1)}(\tilde{k}|x-y|) \quad \mbox{for} ~~N =2\,,
\end{equation}
with $
\tilde{k} = k(a+ib)\varepsilon^{-1/2}.
$

Next, we shall only prove the theorem for the 3D case and the 2D case could be proved in a similar manner.
For $x \in D_{0}$ and $y \in \partial D$, since $|x-y| \geq \delta_0$, it can be verified by straightforward calculations that
\begin{equation}\label{eq:eee3}
\begin{split}
\left|\frac{e^{i \tilde{k}|x-y|}}{4 \pi |x-y|}\right| & \leq \frac{e^{-kb\delta_0 \varepsilon^{-1/2}}}{4\pi \delta_0},\\
\left|\nabla_{y}\frac{e^{i \tilde{k}|x-y|}}{4 \pi |x-y|}\right| & \leq \frac{e^{-kb\delta_0 \varepsilon^{-1/2}}}{4 \pi \delta_0}\left[\frac{k\sqrt{a^2+b^2}}{\varepsilon^{1/2}} + \frac{1}{\delta_0}\right].
\end{split}
\end{equation}
On the other hand, by \eqref{uniform:wellpose2} in Lemma \ref{mainlem} we see that
\begin{equation}\label{eq:e}
\begin{split}
\left\|u_{\varepsilon}^{-}\right\|_{H^{1/2}(\partial D)} & \leq C\varepsilon^{-1/2} \left(\|f\|_{L^{2}(B_{R_{0}} \backslash \overline{\Omega})}+ \|u^{i}\|_{H^{1}(B_{R}\backslash \overline {\Omega})}\right),\\
\left\|\frac{\partial u_{\varepsilon}^{-}}{\partial \nu}\right\|_{H^{-1/2}(\partial D)} & \leq C\varepsilon^{-1/2} \left(\|f\|_{L^{2}(B_{R_{0}} \backslash \overline{\Omega})}+ \|u^{i}\|_{H^{1}(B_{R}\backslash \overline {\Omega})}\right).
\end{split}
\end{equation}
Now using (\ref{eq:eee3}) and (\ref{eq:e}) in (\ref{eq:integral}),
one can obtain (\ref{eq:exponential}) by straightforward calculations.
\end{proof}

\section{A special case and sharpness of convergence estimates}\label{ballcase:check}
In this section, we shall consider a special case of the model system (\ref{eq:transmission problem}):
$D$ is the ball $B_{R_1}$ of radius $R_1$, and  only the subregion $D$ is occupied by
the inhomogeneous medium in the whole space $R^N$, and the rest is the
homogeneous background, so we have $\gamma=1$ and $q=1$ in  (\ref{eq:transmission problem}).
Moreover, we consider the scattering only from plane wave incidence, namely, $f=0$. We shall derive the corresponding estimates of the wave field, which shall demonstrate
the sharpness of our convergence estimates in Section \ref{section:4}.
We will consider only the 3D case while the 2D case could be treated in a similar manner.

{In our current special setting, we can rewrite the equation \eqref{eq:transmission problem} as follows:

Find $u_\varepsilon(x)\in H_{loc}^1(\mathbb{R}^N)$ which solves the system }
\begin{equation}\label{eq:ball}
\begin{cases}
\displaystyle{\nabla\cdot(\varepsilon\nabla u_\varepsilon)+ k ^2(\eta_{0}+i\tau_0)u_\varepsilon=0}\quad & \mbox{in \ $D$},\\
\displaystyle{\Delta u_\varepsilon+k^2  u_\varepsilon=0}\quad & \mbox{in \ $\mathbb{R}^3\backslash\overline{D}$},\\
\ u_\varepsilon(x)=e^{ikx\cdot d}+u_\varepsilon^s(x)\quad & \mbox{in \ $\mathbb{R}^3\backslash\overline{D}$},\\
\ \displaystyle{u_\varepsilon^- = u_\varepsilon^+, \quad\varepsilon\frac{\partial u_\varepsilon^-}{\partial\nu}=\frac{\partial u_\varepsilon^+}{\partial\nu}}\quad & \mbox{on \ $\partial D$},\\
\ \displaystyle{\lim_{|x|\rightarrow \infty}|x|\left\{\frac{\partial u_\varepsilon^s}{\partial |x|}-ik u_\varepsilon^s \right\}=0},
\end{cases}
\end{equation}
and the equation (\ref{eq:soundhard}) with $D$ as a sound-hard obstacle reduces to
\begin{equation}\label{eq:soundhardball}
\begin{cases}
\Delta u + k^{2}u = 0\quad & \mbox{in \ $\mathbb{R}^3\backslash\overline{D}$},\\
u(x)=e^{ikx\cdot d}+u^s(x)\quad & \mbox{in \ $\mathbb{R}^3\backslash\overline{D}$},\\
\displaystyle{\frac{\partial u}{\partial \nu}=0}\quad & \mbox{on \ $\partial D$},\\
\ \displaystyle{\lim_{|x|\rightarrow \infty}|x|\left\{\frac{\partial u^s}{\partial |x|}-ik u^s \right\}=0}.
\end{cases}
\end{equation}

In the sequel, we let $q_{0} = (\eta_{0}+i\tau_{0})/\varepsilon$ and $\sqrt{q_{0}} = \varepsilon^{-1/2}(a+bi)$ with $a>0$, $b>0$. We shall make use of the spherical wave series expansions of the wave fields in (\ref{eq:ball}) and (\ref{eq:soundhardball}), and we refer to \cite{ColKre} for a detailed discussion about spherical wave functions. Let $u_{\varepsilon}(x)$ and $u_{\varepsilon}^{s}$ be given by the following series:
\begin{equation}\label{eq:s1}
\begin{split}
u_{\varepsilon}(x) = & \sum_{n = 0}^{ \infty} \sum_{m = -n}^{n} b_{n}^{m} j_{n}(k\sqrt{q_{0}} |x|) Y_{n}^{m} (\hat{x}), \ \ x \in B_{R_1},\\
u_{\varepsilon}^{s}(x) = & \sum_{n =0}^{\infty} \sum_{m = -n}^{n} a_{n}^{m} h_{n}^{(1)}(k|x|)Y_{n}^{m}(\hat{x}), \ \  x \in  \mathbb{R}^{3} \backslash \overline{B}_{R_1},
\end{split}
\end{equation}
where $\hat{x}=x/|x|$,
and $u^s(x)$ be given by
\begin{equation}\label{eq:s2}
u^s(x)=\sum_{n=0}^\infty\sum_{m=-n}^n c_n^m h_n^{(1)}(k|x|)Y_n^m(\hat{x}),\ \ x\in\mathbb{R}^3\backslash\overline{B}_{R_1}.
\end{equation}
We shall make use of the following series representation of the plane wave
\begin{equation}\label{eq:plane wave}
e^{ik x \cdot d} = \sum_{n = 0}^{\infty} \sum_{m = -n}^{n} i^{n} 4 \pi \overline{Y_{n}^{m}(d)}j_{n}(k|x|)Y_{n}^{m}(\hat{x}).
\end{equation}
By (\ref{eq:s1}) and (\ref{eq:plane wave}), and using the boundary condition on $\partial D$,
we know
\[
c_n^m =  \frac{- i^{n} 4 \pi \overline{Y_{n}^{m}(d)} j_{n}'(kR_{1})}{{h_{n}^{(1)}}'(kR_{1})} .
\]

Next, by the transmission boundary conditions in \eqref{eq:ball} and
comparing the coefficients of $Y_{n}^{m}(\hat{x})$ we derive
\begin{equation}\label{coefficients:hard}
\begin{cases}
&\displaystyle{b_{n}^{m} j_{n}(k \sqrt{q_{0}} R_{1})  = a_{n}^{m} h_{n}^{(1)} (kR_{1}) + i^{n} 4 \pi \overline{Y_{n}^{m}(d)} j_{n}(kR_{1}),} \\
&\displaystyle{\varepsilon k \sqrt{q_{0}} b_{n}^{m} j_{n}'(k\sqrt{q_{0}}R_{1})  = k a_{n}^{m} {h_{n}^{(1)}}'(kR_{1}) + i^{n} k 4 \pi \overline{Y_{n}^{m}(d)} j_{n}'(kR_{1})}.
\end{cases}
\end{equation}
Solving the equation \eqref{coefficients:hard}, we obtain
\begin{equation}\label{coefficients}
\begin{split}
&a_{n}^{m} = \dfrac{i^{n} 4 \pi \overline{Y_{n}^{m}(d)}j_{n}'(kR_{1}) j_{n}(k\sqrt{q_{0}}R_{1}) - \varepsilon\sqrt{q_{0}}i^{n} 4 \pi \overline{Y_{n}^{m}(d)}j_{n}'(k \sqrt{q_{0}}R_{1}) j_{n}(kR_{1})}{\varepsilon\sqrt{q_{0}}j_{n}'(k \sqrt{q_{0}}R_{1})h_{n}^{(1)}(kR_{1}) - {h_{n}^{(1)}}'(kR_{1})j_{n}(k\sqrt{q_{0}}R_{1})}\,, \\
&b_{n}^{m} = \dfrac{-i^{n}4 \pi \overline{Y_{n}^{m}(d)}j_{n}(kR_{1}){h_{n}^{(1)}}'(kR_{1}) +  i^{n} 4 \pi \overline{Y_{n}^{m}(d)}h_{n}^{(1)}(k R_{1})j_{n}'(k R_{1}) }{\varepsilon\sqrt{q_{0}}{j_{n}}'(k \sqrt{q_{0}} R_{1})h_{n}^{(1)}(k R_{1}) - {h_{n}^{(1)}}'(k R_{1})j_{n}(k\sqrt{q_{0}}R_{1})}\,.
\end{split}
\end{equation}

We first consider two wave fields outside $D$ and show the following lemma, which indicates the sharpness
of the estimates  in Theorem~\ref{thm1}.
\begin{lem} For the far field patterns $\mathcal{A}_{\varepsilon}$ and $\mathcal{A}$ corresponding to
the solutions $u_\varepsilon$ and $u$ of systems (\ref{eq:ball}) and (\ref{eq:soundhardball}), we have
\begin{equation}\label{eq:est2}
\left|\mathcal{A}_{\varepsilon}(\hat{x})-\mathcal{A}(\hat{x})\right|
= C_{\cal A}\,\varepsilon^{1/2} +O(\varepsilon),\quad \forall \hat{x}\in\mathbb{S}^2
\end{equation}
where $C_{\cal A}$ depends only on {$\eta_{0}$, $\tau_{0}$, $k$, $R_{1}$, $d$}.
\end{lem}

\begin{proof}
In fact, by (\ref{eq:s1}) and (\ref{eq:s2}), we have
\begin{equation}\label{far:fie}
\begin{split}
\mathcal{A}_{\varepsilon}(\hat{x}) =& \frac{1}{k}\sum_{n =0}^{\infty} \sum_{m = -n}^{n} \frac{1}{i^{n+1}}a_{n}^{m}Y_{n}^{m}(\hat{x}), \\
\mathcal{A}(\hat{x}) = & \frac{i}{k}\sum_{n=0}^\infty 4\pi\frac{j_n'(kR_{1})}{{h_n^{(1)}}'(kR_{1})} \sum_{m=-n}^{n}\overline{Y_{n}^{m}(d)}Y_{n}^{m}(\hat{x}).
\end{split}
\end{equation}
But it follows from (\ref{coefficients}) that
\begin{equation}\label{coeff:hard}
a_{n}^{m} =  \dfrac{i^{n} 4\pi \overline{Y_{n}^{m}(d)}j_{n}'(kR_{1}) - T(q_{0},n)i^{n}4\pi\overline{Y_{n}^{m}(d)}j_{n}(kR_{1})}
{T(q_{0},n)h_{n}^{(1)}(kR_{1})-
{h_{n}^{(1)}}'(k R_{1})}
\end{equation}
with
\[
T(q_{0},n): = \varepsilon\sqrt{q_{0}}\frac{j_{n}'(k\sqrt{q_{0}}R_{1})}{j_{n}(k\sqrt{q_{0}}R_{1})}.
\]
Next, we derive the asymptotic development of $T(q_{0}, n)$ as $\varepsilon \rightarrow 0^+$.
Noting that $j_{n}'(z) = \dfrac{n}{z} j_{n}(z)-j_{n+1}(z)$ (cf. \cite{ColKre}), we see
\[
j_{n}'(k\sqrt{q_{0}}R_{1}) = \dfrac{n}{k\sqrt{q_{0}}R_{1}}j_{n}(k\sqrt{q_{0}}R_{1})-j_{n+1}(k\sqrt{q_{0}}R_{1}),
\]
then
\begin{equation}\label{eq:2}
\begin{split}
T(q_{0},n) = & \varepsilon \sqrt{q_{0}}\left[\frac{n}{k\sqrt{q_{0}}R_{1}}-\frac{j_{n+1}(k\sqrt{q_{0}}R_{1})}{j_{n}(k\sqrt{q_{0}}R_{1})}\right]
=  \frac{n\varepsilon}{kR_{1}} - \varepsilon\sqrt{q_{0}}\dfrac{j_{n+1}(k\sqrt{q_{0}}R_{1})}{j_{n}(k\sqrt{q_{0}}R_{1})}.
\end{split}
\end{equation}

In virtue of the asymptotic behavior of $j_{n}(z)$ {(cf.
9.2.1 and 10.1.1 \cite{AI}) as $|z| \rightarrow \infty$ and $|\text{arg}\ z| < \pi$, } one has
\begin{equation} \label{asymbessel}
{j_{n}(z)  = \dfrac{1}{z} \{\cos(z -  n\pi/2  - \pi/2)  + e^{|\Im z|}\mathcal{O}(|z|^{-1})\}}\\
\end{equation}
 and as $\varepsilon \rightarrow +0$ (cf. \cite{LiLiuSun}), one also has
\begin{equation}\label{var:1}
\frac{j_{n+1}(k\sqrt{q_{0}}R_{1})}{j_{n}(k\sqrt{q_{0}}R_{1})} \sim e^{i\pi/2}.
\end{equation}
Combining (\ref{eq:2})--(\ref{var:1}), one has by direct calculations
\begin{equation}\label{var:2}
\left|T(q_{0},n)\right| \leq \left|\frac{n\varepsilon}{kR_{1}}\right| + \left|\varepsilon\sqrt{q_{0}}\frac{j_{n+1}(k\sqrt{q_{0}}R_{1})}{j_{n}(k\sqrt{q_{0}}R_{1})}\right|  = \mathcal{O}(n\varepsilon + \sqrt{\varepsilon}).
\end{equation}
Now, by \eqref{far:fie}, we have
\begin{equation}\label{eq:33}
\mathcal{A}_{\varepsilon}(\hat{x}) - \mathcal{A}(\hat{x}) = \frac{i}{k}\sum_{n =0}^{\infty} \sum_{m = -n}^{n} \left\{\frac{-1}{i^{n}}a_{n}^{m} - 4\pi\frac{j_n'(kR_{1})}{{h_n^{(1)}}'(kR_{1})}\overline{Y_{n}^{m}(d)}\right\}Y_{n}^{m}(\hat{x}).
\end{equation}
In the sequel, we let
\[
q_{n}^{m} =  \frac{-1}{i^{n}}a_{n}^{m} - 4\pi\frac{j_n'(kR_{1})}{{h_n^{(1)}}'(kR_{1})}\overline{Y_{n}^{m}(d)}.
\]
By using the Wronskian $j_{n}(t)y_{n}'(t)-j_{n}'(t)y_{n}(t)={1}/{t^2}$, we then have
\[
q_{n}^{m}= \frac{iT(q_{0},n)4\pi \overline{Y_{n}^{m}(d)}} {k^2R_{1}^2[T(q_{0},n)h_{n}^{(1)}(kR_{1})-
{h_{n}^{(1)}}'(k R_{1})]{h_{n}^{(1)}}'(k R_{1})}.
\]
Next by the asymptotic behavior of $h_{n}^{(1)}(kR_{1})$ (cf. \cite{ColKre}),
\[
h_{n}^{(1)}(kR_{1}) \sim \frac{1\cdot3\cdots (2n-1)}{i (kR_{1})^{n+1}}(1+ \mathcal{O}(\frac{1}{n})), \quad n \rightarrow +\infty,
\]
and also using the relation ${h_{n}^{(1)}}'(z) = -h_{n+1}^{(1)}(z) + \dfrac{n}{z}h_{n}^{(1)}(z)$, we have
\begin{equation}\label{eq:11}
q_{n}^{m} \sim i\frac{4\pi \overline{Y_{n}^{m}(d)}} {k^2R_{1}^2{h_{n}^{(1)}}'(k R_{1})^2} \frac{\frac{n\varepsilon}{kR_{1}} - \varepsilon\sqrt{q_{0}}e^{i\pi/2}}{\{ (\frac{n\varepsilon}{kR_{1}} - \varepsilon\sqrt{q_{0}}e^{i\pi/2})\frac{-kR_{1}}{n+1}- 1\}}.
\end{equation}

By (\ref{eq:11}) and (\ref{eq:22}), one readily sees that for sufficiently large $n$ and
small $\varepsilon$,
\begin{equation}\label{eq:44}
q_{n}^{m}Y_{n}^{m}(\hat{x}) \sim  -\frac{4\pi\overline{Y_{n}^{m}(d)}Y_{n}^{m}(\hat{x}) } {k^2R_{1}^2{h_{n}^{(1)}}'(k R_{1})^2}  \varepsilon ^{1/2}(a^2 + b^2)^{1/2}+ \mathcal{O}(\varepsilon),
\end{equation}
so {constant $C_{\cal A}$ in (\ref{eq:est2})} can be chosen as
\[
\left|\sum_{n=1}^{\infty}\sum_{m=-n}^{n}\frac{4\pi\overline{Y_{n}^{m}(d)}Y_{n}^{m}(\hat{x}) } {k^3R_{1}^2{h_{n}^{(1)}}'(k R_{1})^2}  (a^2 + b^2)^{1/2}\right|.
\]
{Noting that for any $n,m\in\mathbb{N}$ (cf. \cite{ColKre}),
\begin{equation}\label{eq:22}
 |\overline{Y_{n}^{m}(d)}Y_{n}^{m}(\hat{x})| \leq \frac{2n+1}{4 \pi},
 \end{equation}
hence $C_{\cal A}$ is bounded.}
Finally, using (\ref{eq:44}) and the asymptotic development of ${h_{n}^{(1)}}'(kR_{1})$ for large $n$ (cf. \cite{ColKre}), one can show (\ref{eq:est2}) from (\ref{eq:33}) by direct calculations.
\end{proof}

%\begin{rem}
%We would like to remark that our above estimates are sharp in terms of $\varepsilon$, and this also indicates the sharpness of the estimates of Theorem \ref{thm1}.
%\end{rem}

Next, we consider the normal velocity of the wave field $u_\varepsilon$ on $\partial B_{R_1}$ and {show that there exists a constant $C_{\cal \nu}$ which depends only on $k, R_{1}, d, \eta_{0}, \tau_{0}$ such that}
\begin{equation}\label{eq:normal velocity}
\left\|\frac{\partial u_{\varepsilon}^{+}}{\partial \nu}\right\|_{H^{-1/2}(\partial B_{R_{1}})}
= C_\nu \,\varepsilon^{1/2} +O(\varepsilon).
\end{equation}
Clearly the estimate (\ref{eq:normal velocity})
shows the sharpness of the estimate in Theorem \ref{thm2}.

In fact, by the transmission condition on $\partial B_{R_1}$ we have
\[
\frac{\partial u_{\varepsilon}^{+}}{\partial \nu} = \varepsilon \frac{\partial u_{\varepsilon}^{-}}{\partial \nu}|_{\partial B_{R_{1}}} =  \varepsilon k\sqrt{q_{0} } \sum_{n = 0}^{ \infty} \sum_{m = -n}^{n} b_{n}^{m} j_{n}'(k\sqrt{q_{0}} R_{1}) Y_{n}^{m} (\hat{x}).
\]
Using the Wronskian relation, $j_{n}(t)y_{n}'(t)-j_{n}'(t)y_{n}(t) ={1}/{t^2}$, we get
\begin{equation}\label{series:boundary}
b_{n}^{m}j_{n}'(k\sqrt{q_{0}}R_{1}) = \frac{-i^{n+1}4\pi \overline{Y_{n}^{m}(d)}}{k^2R_{1}^2\{T(q_{0},n)h_{n}^{(1)}(kR_{1}) -{h_{n}^{(1)}}'(kR_{1})\}} \frac{j_{n}'(k\sqrt{q_{0}}R_{1})}{j_{n}(k\sqrt{q_{0}}R_{1})}.
\end{equation}
By direct calculations we obtain
\begin{equation}\label{eq:s3}
\begin{split}
&\left\|\varepsilon \frac{\partial u_{\varepsilon}^{-}}{\partial \nu}\right\|_{H^{-1/2}(\partial B_{R_{1}})}\\
=& \varepsilon|k\sqrt{q_{0}}| \left({\sum_{n = 0}^{ \infty} \sum_{m = -n}^{n}\left(1+\frac{n(n+1)}{R_{1}^2}\right)^{-1/2}|b_{n}^{m} j_{n}'(k\sqrt{q_{0}} R_{1})R_{1}|^2}\right)^{1/2}.
\end{split}
\end{equation}
Then by \eqref{var:2}, \eqref{series:boundary} and the asymptotic behaviors of $h_{n}^{(1)}(kR_{1})$ and ${h_{n}^{(1)}}'(kR_{1})$ for large $n$ (cf. \cite{ColKre}), one can show that the series involved in (\ref{eq:s3})
%\[
%\sum_{n = 0}^{ \infty} \sum_{m = -n}^{n}\left(1+\frac{n(n+1)}{R_{1}^2}\right)^{-1/2}|b_{n}^{m} j_{n}'(k\sqrt{q_{0}} R_{1})R_{1}|^2
%\]
converges to
\[
l_0: = \sum_{n = 0}^{ \infty} \sum_{m = -n}^{n}\left(1+\frac{n(n+1)}{R_{1}^2}\right)^{-1/2}\frac{16\pi^2 |\overline{Y_{n}^{m}(d)}|^2 }{k^4R_{1}^2|{h_{n}^{(1)}}'(kR_{1})|^2}
\]
as $\varepsilon\rightarrow 0^+$.
Hence, for $\varepsilon$ sufficiently small we have
\begin{equation}
\left\|\frac{\partial u_{\varepsilon}^{+}}{\partial \nu}\right\|_{H^{-1/2}(\partial B_{R_{1}})}
=\left\|\varepsilon \frac{\partial u_{\varepsilon}^{-}}{\partial \nu}\right\|_{H^{-1/2}(\partial B_{R_{1}})}
= C_\nu \sqrt{\varepsilon} +O(\varepsilon),
\end{equation}
with {$C_\nu= 2k \sqrt{l_0}(a^2+b^2)^{1/2}$.}

Finally, we consider the wave field $u_\varepsilon$ inside $B_{R_{2}} \Subset B_{R_1} $ with $\delta_0=R_{1} - R_{2}> 0$.
By \eqref{eq:s1}, it suffices for us to consider the asymptotic development of $b_{n}^{m}j_{n}(k\sqrt{q_{0}}|x|)$ for $|x| \leq R_{2}$. We first note that
\begin{equation}\label{series:inter}
\begin{split}
& b_{n}^{m}j_{n}(k\sqrt{q_{0}}|x|) = b_{n}^{m} j_{n}(k\sqrt{q_{0}}R_{1}) \frac{j_{n}(k\sqrt{q_{0}}|x|)}{j_{n}(k\sqrt{q_{0}}R_{1})}\\
=&\frac{-i^{n+1}4\pi\overline{Y_{n}^{m}(d)}}{k^2R_{1}^2\{T(q_{0},n)h_{n}^{(1)}(kR_{1}) - {h_{n}^{(1)}}'(kR_{1})\}} \frac{j_{n}(k\sqrt{q_{0}}|x|)}{j_{n}(k\sqrt{q_{0}}R_{1})}.
\end{split}
\end{equation}
By \eqref{asymbessel} one sees that
\begin{equation}\label{eq:6}
|j_{n}(k\sqrt{q_{0}}R_{1})| \sim \frac{e^{kbR_{1}\varepsilon^{-1/2}}}{R_{1}}\, \quad
\mbox{as} ~~ \varepsilon\rightarrow 0^+\,.
\end{equation}
In the sequel, we consider two separate cases for $u_\varepsilon(x)$ with $x \in B_{R_{2}}$.
{First for the case that }
$|k\sqrt{q_{0}}||x|  =k \varepsilon^{-1/2}|a+ib||x|>1$, then $1/|x| \leq k \varepsilon^{-1/2}|a+ib|$,
and we can show
\begin{equation}\label{eq:7}
\left|\frac{j_{n}(k\sqrt{q_{0}}|x|)}{j_{n}(k\sqrt{q_{0}}R_{1})}\right| \sim
 \frac{R_{1}}{|x|}e^{-kb(R_{1}-|x|)/\sqrt{\varepsilon}}
 \le k R_{1} \varepsilon^{-1/2}|a+ib|e^{-kb\delta_0/\sqrt{\varepsilon}}
\end{equation}
as $\varepsilon \rightarrow 0^+$.
Hence by combining (\ref{eq:22}), (\ref{series:inter}) with (\ref{eq:7})
we derive that
\[
\begin{split}
|u_{\varepsilon}(x)| \leq &  \sum_{n = 0}^{ \infty} \sum_{m = -n}^{n} |b_{n}^{m} j_{n}(k\sqrt{q_{0}} |x|) Y_{n}^{m} (\hat{x})|\\
\leq &
\ \displaystyle{k|a+ib|e^{-kb\delta_0/2\sqrt{\varepsilon}}}\sum_{n = 0}^{ \infty} \sum_{m = -n}^{n} \left|\frac{8\pi\overline{Y_{n}^{m}(d)}Y_{n}^{m}(\hat{x})}{k^2R_{1}{h_{n}^{(1)}}'(kR_{1})}\right|,\\
\leq & M_{1}k|a+ib|e^{-kb\delta_0/2\sqrt{\varepsilon}},\quad \forall x\in B_{R_2}
\end{split}
\]
for sufficiently small $\varepsilon$ such that
$\varepsilon^{-1/2}|a+ib| \exp({-kb\delta_0/(2\sqrt{\varepsilon})}) \leq 1$,
where
\[
M_{1}: = \sum_{n = 0}^{ \infty} \sum_{m = -n}^{n} \left|\frac{2(2n+1)}{k^2R_{1}{h_{n}^{(1)}}'(kR_{1})}\right| < +\infty.
\]
{For the other case, if $|k\sqrt{q_{0}}||x| = k \varepsilon^{-1/2}|a+ib||x| \leq 1$,
then using the asymptotic behavior of $j_{n}(z)$ for large $n$ we know
there exists a constant $M_{2}$ such that
\begin{equation}\label{eq:5}
|j_{n}(k\sqrt{q_{0}}|x|)| \leq M_{2}, \quad \forall n \in \mathbb{N}.
\end{equation}
In a similar manner as we did above one can obtain the following exponentially decay estimate
$$
|u_{\varepsilon}(x)| \leq \sum_{n = 0}^{ \infty} \sum_{m = -n}^{n} \left|\frac{2(2n+1)}{k^2R_{1}{h_{n}^{(1)}}'(kR_{1})}\right|M_{2}e^{-kbR_{1}\varepsilon^{-1/2}}
$$}
as $\varepsilon \rightarrow +0$, by using (\ref{series:inter}), (\ref{eq:6}) and (\ref{eq:5}).
This verifies the sharpness of Theorem \ref{thm3}.

\section*{Appendix}

We shall give a proof of the well-posedness of the scattering problem (\ref{eq:transmission problem}), which was also needed in the proof of Lemma~\ref{mainlem}. We could not find a convenient literature
for the results, so for completeness
we present it in this appendix.
Our argument follows the Lax-Phillips method presented in \cite{Isa}.

Let
\begin{equation}
\{\alpha,\beta\}=\begin{cases}
1, 1\quad &\mbox{in\ $\mathbb{R}^N\backslash\overline{\Omega}$},\\
\gamma, q\quad &\mbox{in\ $\Omega\backslash\overline{D}$},\\
\varepsilon, \eta_0+i\tau_0\quad &\mbox{in\ $D$}.
\end{cases}
\end{equation}
Then the scattering problem (\ref{eq:transmission problem}) can be formulated as follows:

{
Find $u\in H_{loc}^1(\mathbb{R}^N)$ such that
 $u=u^i+u^s$ {in} $\mathbb{R}^N\backslash\overline{\Omega}$
 and solves the equation
\begin{equation}\label{eq:1}
\begin{cases}
\mathscr{L} u:=\nabla\cdot(\alpha\nabla u)+k^2 \beta u=f\quad &\mbox{in\ \ $\mathbb{R}^N$},\\
\ \displaystyle{\lim_{|x|\rightarrow \infty}|x|^{(N-1)/2}\left\{\frac{\partial u^s}{\partial |x|}-ik u^s \right\}=0}
\end{cases}
\end{equation}
}
where we assume $supp(f)\subset B_{R_0}\backslash \Omega$.

The uniqueness of the solutions to the system (\ref{eq:1})
can be shown in a similar argument as the one used in the proof of Lemma~\ref{important:lemma}.
Next we show only the existence and stability estimate.

In the following, by appropriately choosing $R_0$ we can assume that
$k^2$ is not a Dirichlet eigenvalue in $B_{R_0+1}$. Let $\theta(x)\in C^\infty(\mathbb{R}^N)$ be a cut-off function such that $\theta(x)=0$ for $|x|<R_0$ and $\theta(x)=1$ for $|x|>R_0+1$.
{
Setting
\begin{equation}\label{eq:w0}
W=u\quad\mbox{in\ ~$\Omega$} \quad
\mbox{and} \quad
%\end{equation}
%and
%\begin{equation}\label{eq:w00}
W=u^s+(1-\theta)u^i\quad \mbox{in\ ~$\mathbb{R}^N\backslash\overline{\Omega}$},
\end{equation}
we can then verify directly} that $W\in H^1_{loc}(\mathbb{R}^N)$ satisfies
\begin{equation}\label{eq:w}
\begin{cases}
\mathscr{L}W=g \quad \mbox{in\ ~$\mathbb{R}^N$},\\
\ \displaystyle{\lim_{|x|\rightarrow \infty}|x|^{(N-1)/2}\left\{\frac{\partial W}{\partial |x|}-i k W \right\}}=0,
\end{cases}
\end{equation}
with
$
g=-(\Delta+k^2)(\theta u^i)+f\in L^2(B_{R_0+1}\backslash \Omega).
$

Next, we look for a solution to (\ref{eq:w}) of the following form
\begin{equation}\label{eq:w1}
W=w-\phi(w-V),
\end{equation}
where $\phi$ is $C^\infty$ cut-off function such that $\phi=1$ in $B_{R_0}$ and $\phi=0$ in $\mathbb{R}^N\backslash B_{R_0+1}$. We let $V\in H^1(B_{R_0+1})$ be the solution of the system
\begin{equation}
\begin{cases}
\mathscr{L} V=g^*\quad &\mbox{in\ $B_{R_0+1}$},\\
V=0\quad &\mbox{on\ $\partial B_{R_0+1}$}
\end{cases}
\end{equation}
and $w\in H^1_{loc}(\mathbb{R}^N)$ be the solution of the system
\begin{equation}
\begin{cases}
(\Delta+k^2) w=g^*\quad &\mbox{in\ $\mathbb{R}^N$},\\
\displaystyle{\lim_{|x|\rightarrow \infty}|x|^{(N-1)/2}\left\{\frac{\partial w}{\partial |x|}-i k w \right\}}=0,
\end{cases}
\end{equation}
where $g^*\in L^2(B_{R_0+1}\backslash \Omega)$ shall be determined later.

Clearly, by the classical regularity estimates we see
\[
V\in H^2(B_{R_0+1}\backslash\overline{\Omega})\quad\mbox{and}\quad w\in H_{loc}^2(\mathbb{R}^N).
\]
By direct verification we have
\begin{equation}\label{eq:aaa}
\begin{split}
g=&(\Delta+k^2)W= \Delta w+k^2 w+\Delta\phi (w-V)\\
&+ 2\nabla\phi\cdot\nabla(w-V)+\phi\left(\Delta(w-V)+k^2(w-V)\right)\\
=& g^*+K g^*,
\end{split}
\end{equation}
where $K$ is defined to be
$
Kg^*=\Delta \phi(w-V)+2\nabla\phi\cdot\nabla(w-V).
$

We can show that
$K$ is compact from $L^2(B_{R_0+1}\backslash\Omega)$ to itself. We shall make use of the Fredholm theory to show the unique solvability of (\ref{eq:aaa}). It suffices to show the uniqueness of solution to (\ref{eq:aaa}). We set $g=0$. By (\ref{eq:w}) we have $W=0$. Hence $w=\phi(w-V)$ in $\mathbb{R}^N$ and
$V=0$ in $\Omega$ and $w=0$ in $\mathbb{R}^N\backslash B_{R_0+1}$. It is straightforward to verify that
\begin{equation}\label{eq:well}
\begin{cases}
(\Delta+k^2)(V-w)=0\quad & \mbox{in\ \ $B_{R_0+1}$},\\
V-w=0\quad & \mbox{on\ \ $\partial B_{R_0+1}$},
\end{cases}
\end{equation}
hence $V-w=0$. Therefore $w=0$, which then implies that $g^*=0$. Then
by the Fredholm theory
we have a unique $g^*\in L^2(B_{R_0+1}\backslash \Omega)$ to (\ref{eq:aaa}) such that
\[
\|g^*\|_{L^2(B_{R_0+1}\backslash \Omega)}\leq C\|g\|_{L^2(B_{R_0+1}\backslash \Omega)}\leq C\left(\|u^i\|_{H^1(B_{R_0+1}\backslash \overline{\Omega})}+\|f\|_{L^2(B_{R_0}\backslash \Omega)}\right)\,.
\]s
Finally, by {the classical theory on elliptic equations one can show that}
\[
\|u\|_{H^1(B_{R_0+1}\backslash\overline{\Omega})}\leq C\left(\|f\|_{L^2(B_{R_0}\backslash\Omega)}+\|u^i\|_{H^1(B_{R_0+1}\backslash\overline{\Omega})}\right).
\]

\end{document}